\documentclass[12pt]{amsart}

\usepackage{mathptmx,newtxtext,enumerate,newtxmath}

\usepackage[text={6.5in,8in},centering]{geometry} 
\usepackage[dvipsnames]{xcolor}   
\usepackage{xparse}
\usepackage{xr-hyper}
\usepackage[linktocpage=true,colorlinks=true,hyperindex,citecolor=teal,linkcolor=Sepia]{hyperref}

\newtheorem{theorem}{Theorem}[section]

\newtheorem{lemma}[theorem]{Lemma}

\numberwithin{equation}{section}

\begin{document}
	
\author{Amartya Goswami}
\address{[1] Department of Mathematics and Applied Mathematics, University of Johannesburg, South Africa;
\newline
\hspace*{.42cm}[2] National Institute for Theoretical and Computational Sciences (NITheCS), South Africa.}
\email{agoswami@uj.ac.za}
	
\title{Spectrality and monoids}
	
\date{}
	


	
	
\maketitle
	
\begin{abstract}
We prove that the set of proper ideals of a monoid endowed with coarse lower topology is a spectral space.
\end{abstract}
	
\smallskip 

\section{Introduction}
\smallskip 

The aim of this section is to highlight on the study conducted on spectral spaces in relation to diverse structures and a variety of topologies.
The concept of a spectral space was first introduced in \cite{H69}. In this framework, spectral spaces are defined as the spectra of prime ideals of rings, equipped with the Zariski topology. Since its inception, the study of spectral spaces has expanded across various contexts. Several instances illustrate this diversification. Apart from the Zariski topology, multiple other topologies such as Lawson, Scott, lower, upper, ultra-filter, patch, inverse, and others have been explored for prime ideals (for comprehensive insights, we refer our reader to \cite{DST19}).

Among distinct classes of ideals, the characterization of the set of all ideals under the lower topology as a spectral space was established in \cite{P94}. Furthermore, \cite{FFS2016} demonstrated that spaces of all ideals, proper ideals, and radical ideals of a ring, endowed with a hull-kernel-type topology, are also spectral. It is noteworthy that the Zariski-Riemann space of valuation domains within a given field stands as an early example of a specialized spectral space (see \cite{DF86, DFF87, FS20}).

In \cite{FFS16}, investigation into semistar operations (of finite type) led to the identification of their space, equipped with a topology reminiscent of the Zariski topology, as a spectral space. Additionally, \cite{R20} established that the collection of continuous valuations on a topological monoid, with the topology determined by any finitely generated ideal, forms a spectral space.

In the context of module, \cite{AH12} introduced conditions under which prime spectra of modules qualify as spectral spaces. Another pertinent avenue involving spectral spaces is the realm of Stone duality. The prime filters of a distributive lattice, equipped with the Stone topology, constitute a spectral space (see \cite[Chapter 3]{DST19}). Similarly, the proper filters of a Boolean algebra or, more generally, an ortholattice, endowed with a Stone-like topology, exhibit spectral space properties (for references, see \cite{BH20, MY22}).

Although there are intensive studies of various classes of ideals for monoids (or semigroups with identity) have been done (see \cite{Anj81,Aub53,AJ84,Cho77, CP61, DP22, Gri01, Gri50, Kis63, Maj22, PK92, Sch69, Sch76}), to the best of the author's knowledge, the concept of spectrality within the context of monoids has not previously been studied. This endeavour seeks to furnish an exemplification of a spectral space linked to monoids. The proof of our main result is self-contained and remains constructible topology-independent of any specific topology. In establishing the notion of spectrality, we have employed a technique (see Lemma \ref{cso}) that circumvents the need to verify existence condition of certain type of basis of a spectral space. It is worth mentioning that the extension of the algebraic structure to semigroups does not offer a viable generalization when exploring spectrality. Notably, for the fulfilment of quasi-compactness--a condition integral to spectrality--the presence of a multiplicative identity assumes necessary (see \textsection 3, Proof (2)).

To ensure the self-contained nature of this paper, all essential definitions and proofs have been  incorporated.
\smallskip 

\section{Preliminaries} \label{prlm}
\smallskip 

A (multiplicative) \emph{ commutative monoid} is a system $(M, \cdot, 1)$ such that $(M,\cdot)$ is a commutative (multiplicative) semigroup with $1$ as the multiplicative identity. For simplicity, we shall write $a b$ for $a\cdot b$, for all $a,$ $b\in M$. An \emph{ideal of} $M$ is a nonempty subset $I$ of $M$ such that $im\in I$, whenever $i\in I$ and $m\in M$. If $\{I_{\lambda}\}_{\lambda \in \Lambda}$ is a family of ideals of $M$, then so are their intersection $\bigcap_{\lambda \in \Lambda}I_{\lambda}$ and their \emph{sum} $\sum_{\lambda \in \Lambda}I_{\lambda}$(that is, the ideal generated by $\bigcup_{\lambda \in \Lambda}I_{\lambda}$).    The notation $\mathcal{I}_M$ will denote the set of all ideals of $M$. It is easy to see that $\mathcal{I}_M$ is a complete lattice. An ideal $I$ of $M$ is called \emph{proper} if $I\neq M$, and by $\mathcal{I}^+_M$ we shall denote the set of all proper ideals of $M$. We say a proper ideal $J$ of a monoid $M$ is \emph{maximal} if there is no proper ideal of $M$ that properly contains $J$. 

Let $L$ be a complete lattice and $x\in L$. A \emph{cover of} $x$ is a family $\{y_{\lambda}\}_{\lambda\in \Lambda}$ of elements of $L$ such that $x\leqslant \bigvee_{\lambda\in \Lambda}y_{\lambda}$. An element $x$ of $L$ is called \emph{compact} if every cover of $x$ has a finite subcover. A lattice is \emph{algebraic} if it is complete and every element is a least upper
bound of compact elements. It is easy to see that the complete lattice $\mathcal{I}(M)$ is algebraic, for every ideal is a sum of finitely generated (and hence compact) ideals of $M$. This property of $\mathcal{I}(M)$ is going to play a crucial role in our proof of the main result. 

Let us now recall a few relevant terminologies from topology. A space is called \emph{quasi-compact} if every open cover of it has a finite subcover, or equivalently, the space satisfies the finite intersection property. In this definition of quasi-compactness, we do not assume the space is $T_2.$ A closed subset $S$ of a space $X$ is called \emph{irreducible} if $S$ is not the union of two properly smaller closed
subsets of $X$. A  space $X$ is called \emph{sober} if every non-empty irreducible closed subset $\mathcal K$ of $X$ is of the form: $\mathcal K=\mathcal{C}_{\{x\}}$, the closure of an unique singleton set $\{x\}$.
By \cite{H69}, a \emph{spectral space} is a topological space that is  quasi-compact, sober,  admitting a basis of quasi-compact open subspaces that is closed under finite intersections. 

Let $M$ be a monoid. The coarse lower topology on  $\mathcal{I}^+_M$ is the topology for which the sets of the type: 
\[
\mathcal{V}_J=
\left\{I\in \mathcal{I}^+_M\mid J\subseteq I \right\}, \qquad (J\in \mathcal{I}_M);
\]
form a subbasis of closed sets, 
and by $\mathcal{I}^+_M$, we shall also denote the topological space.   
\smallskip 

\section{The Main Theorem}
\smallskip 

We are now prepared to present and proof our main result.

\begin{theorem}\label{mth}
Let $M$ be a monoid. Then the set $\mathcal{I}^+_M$ of proper ideals of $M$ endowed with coarse lower topology is a spectral space.
\end{theorem}

To prove Theorem \ref{mth}, we need to show the following:
\begin{enumerate}
\item The space $\mathcal{I}^+_M$ is quasi-compact.

\item The space $\mathcal{I}^+_M$ is sober.

\item The coarse lower topology on $\mathcal{I}^+_M$ admits a basis of quasi-compact open subspaces that is closed under finite intersections.
\end{enumerate}

Our strategy here is slightly different. Although (1) and (2) are comparatively easy to verify, however, condition (3) is more cumbersome. Thanks to the next lemma, the checking of (3) can be avoided with an expense of showing two additional requirements as we shall see after the proof this lemma.

\begin{lemma}\label{cso}
A quasi-compact, sober, open subspace of a spectral space is spectral. 
\end{lemma}

\begin{proof} ---
Suppose that $S$ is a quasi-compact, sober, open subspace of a spectral space $X$. Since $S$ is quasi-compact and sober, it is sufficient to prove that the set $\mathcal{O}_{S}$ of compact open subsets of $S$ forms a basis of a topology that is closed under finite intersections. It is obvious that a subset $T$ of $S$ is open in $S$ if and only if $T$ is open in $X$, and hence a subset $T$ of $S$ belongs to $\mathcal{O}_{S}$ if and only if $T$ belongs to $\mathcal{O}_{X}.$ Now 
using these facts, we argue as follows.

Let $U$ be an open subset of $S$. Since $U$ is also open in $X$, we have $U=\cup\, \mathcal{U},$ for some subset $\mathcal{U}$ of $\mathcal{O}_{X}.$ But each element of $\mathcal{U}$ being a subset of $U$ is a subset of $S$, and it belongs to $\mathcal{O}_{S}.$ Therefore, every open subset of $S$ can be presented as a union of compact open subsets of $S$. Now it remains to prove that $\mathcal{O}_{S}$ is closed under finite intersections, but this immediately follows from the fact that $\mathcal{O}_{X}$ is closed under finite intersections. 
\end{proof} 
\smallskip

\textit{Proof of Theorem \ref{mth}} ---
It is now clear that to prove the theorem, it is sufficient to check the conditions in Lemma \ref{cso} by taking $X=\mathcal{I}_M$ and $S=\mathcal{I}^+_M$. Therefore, our objective is now to verify the following:

\begin{enumerate}
	
\item \label{ngs}
$\mathcal{I}_M$ is a spectral space;
	
\item \label{oco} $\mathcal{I}^+_M$ is quasi-compact;
	
\item \label{tso} $\mathcal{I}^+_M$ is sober.
	
\item \label{top}  $\mathcal{I}^+_M$ is an open subspace of the space $\mathcal{I}_M$.
\end{enumerate}

(1) Since $\mathcal{I}_M$ is an algebraic lattice (see \textsection \ref{prlm}), the desired spectrality follows from \cite[Theorem 4.2]{P94}.

(2) 
Let  $\{K_{ \lambda}\}_{\lambda \in \Lambda}$ be a family of subbasic closed sets of a normal structure space $\mathcal{I}^+_M$   such that $\bigcap_{\lambda\in \Lambda}K_{ \lambda}=\emptyset.$ Let $\{I_{ \lambda}\}_{\lambda \in \Lambda}$ be a family of ideals of $\mathcal{I}_M$ such  that $\forall \lambda \in \Lambda,$  $K_{ \lambda}=\mathcal{V}_{I_{ \lambda}}.$  Since $\bigcap_{\lambda \in \Lambda}\mathcal{V}_{I_{ \lambda}}=\mathcal{V}_{\bigcup_{\lambda \in \Lambda}I_{ \lambda}},$ we get  $\mathcal{V}_{\bigcup_{\lambda \in \Lambda}I_{ \lambda}}=\emptyset.$ This implies that the ideal $ \sum_{\lambda \in \Lambda}I_{ \lambda}$ must be equal to $M$. Then, in particular, we obtain $1=n_{ \lambda_1}\cdots n_{ \lambda_n},$ where $n_{ \lambda_i}\in I_{\lambda_i}$, for $i=1, \ldots, n$. This implies    $M=\sum_{  i \, =1}^{ n}I_{\lambda_i}.$ Hence,   $\bigcap_{ i\,=1}^{ n}K_{ \lambda_i}=\emptyset,$ and $\mathcal{I}^+_M$ is quasi-compact by Alexander Subbasis Theorem.  	
	
(3) To show the existence of generic points of irreducible closed subsets of $\mathcal{I}^+_M$, it is sufficient to show that $\mathcal{V}_I=\mathcal{C}_I$, whenever $I\in\mathcal{V}_I$. Since $\mathcal{C}_I$ is the smallest closed set containing $I$, and since $\mathcal{V}_I$ is a closed set containing $I$, obviously then  $\mathcal{C}_I\subseteq \mathcal{V}_I$. 
For the reverse inclusion, if $\mathcal{C}_I= \mathcal{I}^+_M$, then 
\[ 
\mathcal{I}^+_M=\mathcal{C}_I\subseteq \mathcal{V}_I\subseteq \mathcal{I}^+_M.
\] 
This proves that $\mathcal{V}_I=\mathcal{C}_I$. Suppose that $\mathcal{C}_I\neq \mathcal{I}^+_M$. Since $\mathcal{C}_I$ is a closed set,  there exists an  index set, $\Lambda$, such that,  for each $\lambda\in\Lambda$, there is a positive integer $n_{\lambda}$ and ideals $I_{\lambda 1},\dots, I_{\lambda n_\lambda}$ of $M$ such that 
\[
\mathcal{C}_I={\bigcap_{\lambda\in\Lambda}}\left({\bigcup_{ i\,=1}^{ n_\lambda}}\mathcal{V}_{I_{\lambda i}}\right).
\]
Since  
$\mathcal{C}_I\neq \mathcal{I}^+_M,$ we  assume that ${\bigcup_{ i\,=1}^{ n_\lambda}}\mathcal{V}_{I_{\lambda i}}$ is non-empty for each $\lambda$. Therefore, $I\in   {\bigcup_{ i\,=1}^{ n_\lambda}}\mathcal{V}_{I_{\lambda i}}$ for each $\lambda$, and hence \[\mathcal{V}_I\subseteq {\bigcup_{ i=1}^{ n_\lambda}}\mathcal{V}_{I_{\lambda i}},\] that is, $\mathcal{V}_I\subseteq \mathcal{C}_I$ as desired. 
To obtain the uniqueness of the generic point, it is sufficient to prove that $\mathcal{I}^+_M$ is a $T_0$-space. Let $I$ and $I'$ be two distinct elements of $\mathcal{I}^+_M$. Then, without loss of generality, we may assume that $I\nsubseteq I'$. Therefore $\mathcal{V}_I$ is a closed set containing $I$ and missing $I'$. 

(4) By considering coarse lower topology on $\mathcal{I}_M$, we immediately obtain \[M=\mathcal{V}(M)=\mathcal{C}_M,\] and hence $\mathcal{I}_M \setminus\mathcal{I}^+_M$ is closed. This  implies $\mathcal{I}^+_M$ is open as required. \hfill \qedsymbol

\end{document}